\documentclass[reqno, 12pt]{amsart}

\usepackage{amssymb, amsmath, amsthm}
\usepackage{graphicx, enumerate}
\usepackage{color}
\usepackage{tikz}
\usepackage[normalem]{ulem}

\allowdisplaybreaks

\newtheorem{Theorem}{Theorem}[section]

\newtheorem{Lemma}[Theorem]{Lemma}

\theoremstyle{definition}

\theoremstyle{remark}

\newtheorem*{Remark}{Remark}

\numberwithin{equation}{section}

\usepackage{color}
\usepackage[bookmarks=false]{hyperref}
\hypersetup{
    colorlinks=true, 
    linktoc=all,     
    linkcolor=blue,  
    citecolor=blue,
    urlcolor =magenta
}

\title[An easy proof of Ramanujan's famous congruences]{An easy proof of 
Ramanujan's famous congruences $p(5m+4)\equiv 0 \equiv \tau(5m+5) \pmod 5$}

\author[H.~S.~Bal]{Hartosh Singh Bal
}
\address{The Caravan\\
Jhandewalan Extn., New Delhi 110001, India}
\email{hartoshbal@gmail.com}

\author[G.~Bhatnagar]{Gaurav Bhatnagar
}
\address{RamanujanExplained.org, 18 Chitra Vihar, Delhi 110092, India}
\email{bhatnagarg@gmail.com}
\urladdr{https://gauravbhatnagar.org}

\date{}


\begin{document}

\begin{abstract} We present a proof of Ramanujan's congruences 
$$p(5n+4) \equiv 0 \pmod 5  \text{ and } \tau(5n+5) \equiv 0 \pmod 5.$$ The proof only requires a limiting
case of Jacobi's triple product, a result that Ramanujan knew well, and a technique which Ramanujan used himself to compute values of $\tau(n)$.
\end{abstract}

\maketitle

\thispagestyle{myheadings}
\font\rms=cmr8
\font\its=cmti8
\font\bfs=cmbx8

\markright{\its S\'eminaire Lotharingien de
Combinatoire \bfs 93 \rms (2025), Article~B93a\hfill}
\def\thepage{}

\section{Introduction}
Let $p(n)$ be the number of unordered partitions of a non-negative integer $n$. They are given by the generating function
$$P(q) =  \sum_{n=0}^\infty p(n) q^n = \prod_{k=1}^\infty \frac{1}{1-q^{k}}=\frac{1}{(1-q)(1-q^2)(1-q^3)\cdots}.$$
Similarly, Ramanujan's $\tau$ function is defined by the generating function
$$\sum_{n=0}^\infty \tau(n+1)q^n=\prod_{k=1}^\infty (1-q^k)^{24}.$$
The objective of this paper is to give a proof of Ramanujan's congruences
\begin{equation*}
p(5m+4)  \equiv 0 \pmod 5 \text{ and }
\tau(5m+5) \equiv 0\pmod 5 .
\end{equation*}
Ramanujan's proof of the $p(5m+4)$ congruence appears in \cite[Paper~25]{Ramanujan-CW}; for his proof of the second congruence, see Berndt and Ono~\cite{BO1999}. Our proof proves both at the same time; even so, it is easier.

We embed these in an infinite family of congruences. 
Let $P_r(n)$ be defined using the equation
\begin{equation}\label{Pr-def}
P(q)^r = \sum_{n=0}^\infty P_r(n)q^n,
\end{equation}
where $r$ is a rational number. In this notation $P_1(n)=p(n)$ and $P_{-24}(n)= \tau(n+1)$. 
We will show:
\begin{equation}\label{th:congruence5}
 P_r(5m+4) \equiv 0 \pmod 5, \text{ if } r \equiv 1 \pmod 5,
\end{equation}
where $r$ is a rational number.\footnote{In these congruences,
a rational number $a/b$, with $a$ and~$b$ relatively prime and $b$ not
divisible by~$5$, is understood as $a\cdot b^{-1}$, where $b^{-1}$ is
the multiplicative inverse of~$b$ modulo~$5$.}
The cases $r=1$ and $r=-24$ give Ramanujan's congruences. 

The main ingredient of our proof is Jacobi's identity, a 
limiting
case of his triple product identity~\cite[p.~14]{Berndt2006},
\begin{equation}\label{jacobi1} 
P(q)^{-3}=\prod_{k=1}^\infty (1-q^k)^{3} = \sum_{k=0}^\infty (-1)^k (2k+1) q^{\frac{k(k+1)}{2}}.
\end{equation}
Ramanujan had rediscovered Jacobi's result. In our notation, it can be rewritten as follows:
\begin{equation}\label{jacobi2}
P_{-3}(n)= 
\begin{cases}
(-1)^k(2k+1), & \text{if } n = \frac{k(k+1)}2;\\
0, & \text{otherwise}.
\end{cases}
\end{equation}

\pagenumbering{arabic}
\addtocounter{page}{1}
\markboth{\SMALL H. S. BAL AND G. BHATNAGAR}
{\SMALL AN EASY PROOF OF RAMANUJAN'S FAMOUS CONGRUENCES}

We mention that the result \eqref{th:congruence5}, and this proof, appears in the authors' previous paper~\cite{BB2022a}. The objective of this note is to present  the bare essentials of this proof. 

\section{The proof}
We first prove a recurrence relation satisfied by the coefficients of powers of any generating function. 
\begin{Lemma} Let $P(q)$ be any power series, $r$ and $s$ be non-zero, real, numbers, and $P_r(n)$  defined by~\eqref{Pr-def}. 
Then we have
\begin{equation}\label{lemma:1}
\sum_{k=0}^n \Big(n-\big( r/s + 1\big)k \Big) P_r(n-k)P_s(k) =0.
\end{equation}
\end{Lemma}
\begin{proof} We take the log derivatives on both sides of~\eqref{Pr-def}, and multiply by $q$, to obtain
$$r \Big(q\frac{d}{dq} \log P(q) \Big) \sum_{n=0}^\infty P_r(n)q^n = \sum_{n=0}^\infty nP_r(n)q^n. $$
Similarly, we have
$$s \Big(q\frac{d}{dq} \log P(q) \Big) \sum_{n=0}^\infty P_s(n)q^n = \sum_{n=0}^\infty nP_s(n)q^n. $$
This gives, on eliminating the common factor,
$$
\frac{r}{s}  \Big(\sum_{n=0}^\infty P_r(n)q^n\Big)\Big(\sum_{n=0}^\infty nP_s(n)q^n\Big) 
=
\Big(\sum_{n=0}^\infty nP_r(n)q^n\Big)\Big(\sum_{n=0}^\infty P_s(n)q^n\Big)
$$
The recurrence \eqref{lemma:1} follows by comparing coefficients of $q^n$ on both sides. 
\end{proof}
\begin{Remark} In this lemma, $P(q)$ can be any generating function, and  $r$ and $s$ any non-zero real or complex numbers, as long as $P(q)^r$ and $P(q)^s$ are formal power series. 
However, we will only apply it when $P(q)$ is the generating function for partitions, and when $r$ and $s$ are 
rational numbers.
\end{Remark}

\begin{proof}[Proof of \eqref{th:congruence5}]
We use an inductive argument using the recurrence relation~\eqref{lemma:1}, with $s=-3$, in the form
\begin{equation}\label{ramanujan-gen-2}
nP_r(n) 
=\sum_{j=1}^\infty (-1)^{j+1}(2j+1)\Big(n+\big({r}/{3}-1\big){j(j+1)}/2\Big) P_r\big(n-{j(j+1)}/2\big).
\end{equation}
Here we have used Jacobi's result in the form~\eqref{jacobi2} for $P_{-3}(k)$.

Take $n=5m+4$. For $m=0$, \eqref{ramanujan-gen-2} reduces to
$$4P_r(4)=(9+r)P_r(3)-5(r+1)P_r(2),$$
so $P_r(4) \equiv 0$ (mod $5$) if $r\equiv 1$ (mod $5$). 

For $m>0$, 
consider the general term 
$$(-1)^{j+1}(2j+1)\left(n+(r/3-1)j(j+1)/2\right) P_r\big(n-j(j+1)/2\big)$$
in \eqref{ramanujan-gen-2} for each $j$. 
When $j\equiv 1$ (mod $5$) it is of the form
$$(-1)^{j+1}  (9+r) P_r(**) \text{ (mod $5$)},$$
and, when $j\equiv 3$ (mod $5$), it is of the form
$$(-1)^{j+1} 2 (2r-2) P_r(**) \text{ (mod $5$)}.$$
So, when $r\equiv 1$, these terms are $0 \pmod 5$. 
When $j\equiv 2$ (mod $5$), then the expression is of the form
$$(-1)^{j+1} 5 (1+r) P_r(**) \text{ (mod $5$)};$$
clearly, it is divisible by $5$. 
Finally,  when $j \equiv 4, 5$ (mod $5$), this reduces to an expression of the form
$$\big( * * *\big) P_r(5m+4-5k),$$
for some $k>0$, and so is $\equiv 0$ (mod $5$), by the induction hypothesis.

Thus each term of the sum on the right-hand side of~\eqref{ramanujan-gen-2} is $0 \pmod 5$. This completes the proof
by induction.
\end{proof}
The following companions to \eqref{th:congruence5} can be obtained from~\eqref{ramanujan-gen-2} by a similar proof: 
\begin{align}
P_r(5m+1) &\equiv 0 \pmod 5 \text{ if } r \equiv 0 \pmod 5;\\
P_r(5m+2) &\equiv 0 \pmod 5 \text{ if } r \equiv 2 \pmod 5;\\
P_r(5m+3) &\equiv 0 \pmod 5 \text{ if } r \equiv 4 \pmod 5.
\end{align}
For further such congruences modulo $3$, and some related results, see~\cite{BB2022a, BB2024}. 

\section{Commentary}
Taking log derivatives to obtain recurrences from generating functions is a standard procedure in generatingfunctionology,  explained in Chapter~1 of Wilf~\cite{wilf2006}. It was one of Ramanujan's favorite tricks. For example, we find the following recurrence in Ramanujan's work
(see \cite[p.~108]{Berndt1994}):
\begin{equation*}
\sum_{d=1}^{n} \sigma(d) p(n-d)=np(n).
\end{equation*}
Here $\sigma(n)$ is the sum of divisors of $n$. This can be obtained by log differentiation of $P(q)$ and comparing coefficients. An extension was given by Ford~\cite{ford1931}. 
See also Entries~12 and~13 in Chapter~10 of Ramanujan's notebooks~\cite[pages~28--32]{Berndt1989} for some intricate examples illustrating Ramanujan's use of log differentiation.  

Ramanujan's own recurrence for $\tau(n)$ is \eqref{ramanujan-gen-2} (with $r=-24$); from this he computed 30 values of $\tau(n)$ in his seminal paper~\cite{SR1916}. His proof uses log derivatives, and is along the lines of our proof of~\eqref{lemma:1}. Gould~\cite{Gould1974} credits an equivalent form of~\eqref{lemma:1}  to Rothe~(1793). 
Lehmer~\cite{Lehmer1943}  used Ramanujan's ideas to compute more values of $\tau(n)$. 

Interestingly, Ramanujan too considered congruence results for $p(n)$ and $\tau(n)$ in tandem (see~\cite{BO1999}); his proof of the $\tau(5n+5)$ congruence uses the $p(5n+4)$ congruence. 
Here is a relation connecting $\tau(n)$ with $p(n)$,  obtained by taking $s=1$ and $r=-24$ in~\eqref{lemma:1}:
\begin{equation}
\sum_{k=0}^n (n+23k) \tau(n-k+1)p(k) =0.
\end{equation}


\begin{thebibliography}{10}

\bibitem{BB2022a}
H.~S. Bal and G.~Bhatnagar.
\newblock The {P}artition-frequency enumeration matrix.
\newblock {\em Ramanujan J.}, 59:51--86, 2022.

\bibitem{BB2024}
H.~S. Bal and G.~Bhatnagar.
\newblock Glaisher's divisors and infinite products.
\newblock {\em J. Integer Seq.}, 27(1):Paper No. 24.1.6, 22, 2024.

\bibitem{Berndt1989}
B.~C. Berndt.
\newblock {\em Ramanujan's notebooks. {P}art {II}}.
\newblock Springer--Verlag, New York, 1989.

\bibitem{Berndt1994}
B.~C. Berndt.
\newblock {\em Ramanujan's notebooks. {P}art {IV}}.
\newblock Springer--Verlag, New York, 1994.

\bibitem{Berndt2006}
B.~C. Berndt.
\newblock {\em Number theory in the spirit of {R}amanujan}, volume~34 of {\em
  Student Mathematical Library}.
\newblock American Mathematical Society, Providence, RI, 2006.

\bibitem{BO1999}
B.~C. Berndt and K.~Ono.
\newblock Ramanujan's unpublished manuscript on the partition and tau functions
  with proofs and commentary.
\newblock The Andrews Festschrift (Maratea, 1998).
\newblock {\em S\'em. Lothar. Combin.}, 42:Art. B42c, 63, 1999.

\bibitem{ford1931}
W.~B. Ford.
\newblock Two theorems on the partitions of numbers.
\newblock {\em Amer. Math. Monthly}, 38(4):183--184, 1931.

\bibitem{Gould1974}
H.~W. Gould.
\newblock Coefficient identities for powers of {T}aylor and {D}irichlet series.
\newblock {\em Amer. Math. Monthly}, 81:3--14, 1974.

\bibitem{Lehmer1943}
D.~H. Lehmer.
\newblock Ramanujan's function {$\tau(n)$}.
\newblock {\em Duke Math. J.}, 10:483--492, 1943.

\bibitem{Ramanujan-CW}
S.~Ramanujan.
\newblock {\em Collected papers of {S}rinivasa {R}amanujan}.
\newblock AMS Chelsea Publishing, Providence, RI, 2000.
\newblock Edited by G. H. Hardy, P. V. Seshu Aiyar and B. M. Wilson, Third
  printing of the 1927 original, With a new preface and commentary by Bruce C.
  Berndt.

\bibitem{SR1916}
S.~Ramanujan.
\newblock On certain arithmetical functions [{Trans.\ Cambridge Philos.\ Soc.
  {\bf 22} (1916), 159--184]}.
\newblock In {\em Collected papers of {S}rinivasa {R}amanujan}, pages 137--162.
  AMS Chelsea Publ., Providence, RI, 2000.

\bibitem{wilf2006}
H.~S. Wilf.
\newblock {\em generatingfunctionology}.
\newblock A K Peters, Ltd., Wellesley, MA, third edition, 2006.

\end{thebibliography}
\end{document}